\documentclass[11pt]{amsart}
\usepackage{amsmath,amssymb}
\usepackage[bookmarks=true]{hyperref}
\newtheorem{theorem}{Theorem}[section]
\newtheorem{prop}[theorem]{Proposition}
\newtheorem{cor}[theorem]{Corollary}

\newtheorem{lemma}[theorem]{Lemma}
\theoremstyle{definition}

\newtheorem*{claim*}{Claim}
\newtheorem*{problem*}{Problem}

 \newcommand{\R}{\mathbb{R}}

\newcommand{\p}{\partial}

\newcommand{\lp}{\Delta}

\newcommand{\ra}{\rightarrow}

\DeclareMathOperator{\diam}{diam}
\DeclareMathOperator{\vol}{vol}

\begin{document}
\title[Effective Volume Growth]{Effective Volume Growth of Three-Manifolds with Positive Scalar Curvature}
\author{Yipeng Wang}
\address{Columbia University \\ 2990 Broadway \\ New York NY 10027 \\ USA}
\begin{abstract}
    In this note, we prove an effective linear volume growth for complete three-manifolds with non-negative Ricci curvature and uniformly positive scalar curvature. This recovers the results obtained by Munteanu-Wang \cite{munteanu2022geometry}. Our method builds upon recent work by Chodosh-Li-Stryker \cite{CLS}, which utilizes the technique of $\mu$-bubbles and the almost-splitting theorem by Cheeger-Colding.
\end{abstract}

\maketitle
\section{Introduction}
The Bishop-Gromov volume comparison theorem asserts that a complete Riemannian manifold $M^n$ with non-negative Ricci curvature exhibits at most Euclidean volume growth: specifically, there exists a universal constant $C(n)$ such that $\vol(B_r(p))\le C r^n$ for any point $p\in M$ and all $r>0$. A well known conjecture of Gromov \cite{Gromov-Lecture} proposes that if $M^n$ additionally possesses uniformly positive scalar curvature $R\ge 1$, then there should be a universal constant $C(n)$ such that $\vol(B_r(p))\le Cr^{n-2}$.

In this note, we explore Gromov's conjecture within the three-dimensional setting to establish an effective linear volume growth result. It is important to mention that, according to Yau's linear volume growth theorem for manifolds with non-negative Ricci curvature \cite{Yau-linear}, Gromov's conjecture would provide a precise characterization of volume growth for three-manifolds with $\text{\rm Ric}_g\ge 0$ and $R\ge 1$.
\begin{theorem}{\label{thm:effective}}
    There exists a universal constant $C$ such that the following statement is true. Let $(M,g)$ be a complete Riemannian 3-manifold with $\text{\rm Ric}_g\ge 0$. For all $p\in M$ and $r>0$, if $R_g\ge 1$ in $B_r(p)$, then 
    $$\vol(B_r(p))\le Cr$$
\end{theorem}
The significant work on Theorem \ref{thm:effective} was initially conducted by Munteanu and Wang \cite{munteanu2022geometry}, who assumed $R_g\ge 1$ throughout the entire manifold $M$. Their analysis focused on the level sets of harmonic functions. Subsequently, Chodosh, Li, and Stryker \cite{CLS} provided an alternative approach for cases where $M$ is non-compact. Employing the technique of of $\mu$-bubbles along with the Cheeger-Colding almost splitting theorem, they have demonstrated that
\[
{\vol(B_r(p))}\le C(p,M,g)r
\]
for all $p\in M$ and $r>0$, also assuming $R_g\ge 1$ over $M$, where the constant may depend on the manifold. In this context, Theorem \ref{thm:effective} can be considered as an effective and localized version of the main results considered in \cite{munteanu2022geometry} and \cite{CLS}. 

After this paper was written, we learned of recent contributions by Wei, Xu and Zhang \cite{wei-xu-zhang}, as well as by Antonelli and Xu \cite{antonelli-xu}, who independently provided asymptotically sharp estimate. We refer some other related works for this problem in higher dimensions \cite{wang2023positive},\cite{Zhu+2022+225+246}.

\subsection*{Acknowledgement} The author wishes to express sincere gratitude to his advisor, Simon Brendle, for inspiring discussions and continuing support. Additionally, the author is thankful to Chao Li for explaining the ideas presented in \cite{CLS}.
\section{Main Ingredients of the Proof}
Now let us describe the main techniques employed in the proof of Theorem \ref{thm:effective}. The method we use is similar to that in \cite{CLS}. The Ricci curvature condition enables the application of the Cheeger-Colding almost splitting theorem \cite{Cheeger-Colding}, which ensures that geodesic balls up to a certain scale, centered at the midpoint of a long geodesic, is Gromov-Hausdorff close to a ball with same radius in $N\times \R$, where $N$ as a length space can be constructed as the level set of certain harmonic function that is closed to the distance function. Additionally, the scalar curvature condition allows us to construct a specific surface, $\Sigma$ (called the $\mu$-bubble) around $N$, so that each connected component of $\Sigma$ maintains a uniform diameter bound. We should note that $\Sigma$ generally possesses numerous connected components, but a more careful analysis of the Cheeger-Colding estimates ensures that $N$ is closed to one specific component of $\Sigma$, therefore having a uniform diameter bound.

By considering the universal cover of $M$, we could assume that $M$ is simply connected. The first key ingredient in our analysis is the geometric estimates of $\mu$-bubbles, which we outline in the following Lemma. The concept of $\mu$-bubble was initially introduced by Gromov, and we refer to \cite{four-lectures} for a general introduction to this technique. 

Throughout this note, we will write
\[
N_R(\Gamma):=\{x:d(x,\Gamma)<R\}
\]
to denote the tubular neighborhood of a given closed subset $\Gamma$ and
\[
\beta_{a,b}:=\beta([a,b])\subset M
\]
for any unit speed minimizing geodesic $\beta$ and $0\le a <b<\infty$.
\begin{lemma}[Chodosh-Li \cite{aspherical},\cite{CLS}, Chodosh-Li-Stryker 
\cite{chodosh2022complete}]{\label{lem:mu-bubble}}
    There exists constants $L$ and $c$ such that the following is true:

Let $(X^3,g)$ be a 3-manifold with boundary. If there exists some $p\in X$ such that $d_X(p,\p X)>L$ and $R_g\ge 1$ in $N_{L}(\p X)$, then there exists an open subsets $\Omega\subset N_{\frac{L}{2}}(\p X)\cap X$ and a smooth surface $\Sigma$ with $\p\Omega= \Sigma\sqcup \p X$ and each connected component of $\Sigma$ has diameter bounded by $c$.
\end{lemma}
Throughout the remainder of this note, we will let $L$ and $c$ to be the universal constants from Lemma \ref{lem:mu-bubble}, and without loss of generality, we may assume that ${L>4c}$.

Fix a point $p\in M$ and a large constant $\ell$, we assume that $R_g\ge 1$ within $B_{\ell}(p)$. Let $\gamma:[0,\ell]\ra M$ to be a unit speed minimizing geodesic with $\gamma(0)=p$. For each $k\ge 1$ where $(k+1)L<\ell$, we apply Lemma \ref{lem:mu-bubble} to $M\backslash \overline{B_{kL}(p)}$. This yields a smooth surface $\Tilde{\Sigma}_k$ that is homologous to $\p B_{kL}(p)$, with the following properties: 
\begin{itemize}
    \item $\Tilde{\Sigma}_k\subset N_{\frac{L}{2}}(\p B_{kL}(p))\cap \left(M\backslash \overline{B_{kL}(p)}\right)$.
    \item The diameter of each connected component of $\Tilde{\Sigma}_k$ is bounded by $c$.
\end{itemize}
Note that for all positive integers $k$ with $(k+1)L<\ell$, we must have $\gamma\cap\Tilde{\Sigma}_k\ne\emptyset$. We set $\Sigma_k\subset \Tilde{\Sigma}_k$ as a connected component that intersects with $\gamma$. Finally, we let $t_k\in [0,\ell] $ such that $p_k:=\gamma(t_k)\in \Sigma_k$.
\begin{lemma}{\label{lem:break-down}}
    We have $d(p_k,p_{k+1})=|t_k-t_{k+1}|< 2L$.
\end{lemma}
\begin{proof}
    Given that $\Sigma_k\subset N_{\frac{L}{2}}(\p B_{kL}(p))\cap \left(M\backslash \overline{B_{kL}(p)}\right)$, it follows that
    $$kL \le t_k < (k+1)L.$$
    Therefore $|t_k-t_{k+1}|<2L$.
\end{proof}
\begin{lemma}{\label{lem:separate}}
    For any $0<s_0<t_k <s_1<\ell$, if $\gamma'$ is any continuous path joining $\gamma(s_0)$ and $\gamma(s_1)$, then $\gamma'\cap\Sigma_k\ne\emptyset$.
\end{lemma}
\begin{proof}
    Suppose not, then $\gamma_{s_0,s_1}$ and $\gamma'$ form a loop that has a non-trivial intersection number with $\Sigma_k$, contradicting the fact that $M$ is simply connected.
\end{proof}
\subsection{Results from Cheeger-Colding Theory}
The second key component of our proof is the almost-splitting theorem, which we outline as follows. Let ${S=5c+2L}$. We let $r_0\gg S$ as a large constant to be determined later. We denote $\Psi=\Psi(R):\mathbb{R}^{+}\to \mathbb{R}$ as a continuous function that may vary from line by line, with the property that $\lim_{R\to\infty}\Psi(R)=0$. 

Let $R\ge r_0$ and $k\in \mathbb{N}^+$ with $R<t_k<\ell-R$.  We set up some notations: First we define $B_k:=B_S(p_k)$ as the region to apply the almost splitting theorem. Let $b$ to be the Buseman function associate with the geodesic $\gamma$
    $$b(y):=d(y,\gamma(t_{k}+ R))-R$$
We then define $h$ as the harmonic replacement of $b$ in a larger ball:
\begin{align*}
    \begin{cases}
        \lp h=0, \text{ in }B_{16S}(p_k)\\
        h=b, \text{ on }\p B_{16S}(p_k)
    \end{cases}
\end{align*}
By perturbing an arbitrarily small amount along $\gamma$, we assume $h(p_k)$ is a regular value of $h$. We define $\Gamma_k:= h^{-1}\{(h(p_k)\}$ to be the level set of $h$ at $h(p_k)$. For any $x\in B_k$, we denote $x'$ to be a point in $\Gamma_k$ that minimizes the distance to $x$ among all points in $\Gamma_k$.
\begin{lemma}[Cheeger-Colding \cite{Cheeger-Colding}]{\label{lem:CC}}
Suppose $x,y,z\in B_{2S}(p_k)$, with $h(x)=h(z)$, and $z$ minimizes the distance from $y$ over the level set $h^{-1}\{h(z)\}$, then
\begin{align}
    {\label{eqn:h-closeb}}
        &|h(x)-b(x)|\le \Psi\\
        &{\label{eqn:h-approximate-distance}}
        \left|d(y,z)-|h(y)-h(z)|\right|\le \Psi\\
        &{\label{eqn:pythagoras}}
        \left|d(x,y)^2-d(x,z)^2-d(y,z)^2\right|\le \Psi
\end{align}
\end{lemma}
\begin{cor}{\label{cor:CC-on-geo}}
For all $t$ with $x=\gamma(t)\in B_{k}$, we have
\begin{align*}
    d(x,x')\le |t-t_k|+\Psi,\qquad d(p_k,x')\le \Psi
\end{align*}
\end{cor}
\begin{proof}
    It is clear that for all $x\in B_k$ we have $x'\in B_{2s}(p_k)$. It then follows from \eqref{eqn:h-closeb} that $|h|\le \Psi$ on $\Gamma_k$. With \eqref{eqn:h-approximate-distance} together we obtain
    \begin{align*}
        \left|d(x,x')-|t-t_k|\right|&= \left|d(x,x')-|b(x)|\right|\\
            &\le \left|d(x,x')-|h(x)|\right|+\Psi\\
            &\le \left|d(x,x')-|h(x)-h(x')|\right|+\Psi\\
            &\le \Psi
    \end{align*}
    This establishes the first inequality; the second inequality then follows from \eqref{eqn:pythagoras}.
\end{proof}
\begin{lemma}{\label{lem:CLS-step-1}}
    There exists some $r_0=r_0(c,L,S)$ such that for all $R\ge r_0$ the following statement is true. For all $k$ with $R<t_k<\ell-R$, we have  
    $$\diam\left(\Gamma_k\cap B_k\right)\le 3c$$
\end{lemma}
\begin{proof}
Suppose, instead, that $\diam (\Gamma_k\cap B_k)>3c$.
Then, there exists some $y\in \Gamma_k\cap B_k$ such that $d(y, B_c(p_k))>2c$.
In particular, since $\diam(\Sigma_k)\le c$, we obtain that $d(y, \Sigma_k)>2c$. 

Now we take $\sigma_{\pm}$ to be the unit speed minimizing geodesic joining $\gamma\left(t_k\pm\frac{S}{2}\right)$ to $y$. By Lemma \ref{lem:separate}, we must have at least one of $\sigma_{\pm}$ intersect with $\Sigma_k$. For simplicity, we denote $\sigma$ as the minimizing geodesic such that
$$\sigma(s_0)=\gamma(\Bar{t})=q,\qquad \sigma(s_1)\in\Sigma_{k},\qquad\sigma(s_2)=y$$
where  $s_0<s_1<s_2$ and $|\bar{t}-t_k|=\frac{S}{2}$.
By Corollary \ref{cor:CC-on-geo}, we know that
\[d(q,q')\le\frac{S}{2}+\Psi,\qquad d(q',p_k)\le \Psi\]
Next, we apply \eqref{eqn:pythagoras} to obtain
\begin{align*}
    (s_2-s_0)^2&\le d(q,q')^2+d(q',y)^2+\Psi\\
    &\le \left(\frac{S}{2}+\Psi\right)^2+\bigg(d(p_k,y)+d(q',p_k)\bigg)^2+\Psi\\
    & \le \left(\frac{S}{2}\right)^2+ d(p_k,y)^2+\Psi
\end{align*}
Hence, by the definition of $\Psi$, there exists some $r_0=r_0(c,L,S)$ such that for all $R\ge r_0$, we have
\begin{equation}{\label{eqn:pythagoras-s2-s0-1}}
\begin{split}
    (s_2-s_0)^2\le d(p_k,y)^2+\left(\frac{S}{2}\right)^2+2c^2
\end{split}
\end{equation}
On the other hand, since $d\left(\sigma(s_1),p_k\right)\le c$, we apply the triangle inequality.
\begin{align*}
    (s_2-s_0)^2
    &=\left[d(q,\sigma(s_1))+d(\sigma(s_1),y)\right]^2\\
    &\ge \left[d(q,p_k)+d(y,p_k)-2c\right]^2\\
    &=d(q,p_k)^2+d(y,p_k)^2+4c^2\\ &+d(q,p_k)\left(d(y,p_k)-2c\right)+d(y,p_k)\left(d(q,p_k)-2c\right)
\end{align*}
Given that $d(y,p_k)\ge 2c$ and $d(q,p_k)=\frac{S}{2}\ge  2c$, this implies
\begin{align*}
    (s_2-s_0)^2
    &\ge d(q,p_k)^2+d(y,p_k)^2+4c^2\\
    &=\left(\frac{S}{2}\right)^2+d(y,p_k)^2+4c^2
\end{align*}
which contradicts \eqref{eqn:pythagoras-s2-s0-1}.
\end{proof}
\begin{lemma}{\label{lem:ball-around-geodesic}}
    There exists some $r_0=r_0(c,L,S)$ such that for all $R\ge r_0$ the following statement is true. For all $k$ with $R<t_k<\ell-R$, we have $B_k\subset N_{4c}(\gamma_{0,\ell})$. 
\end{lemma} 
\begin{proof}
Suppose that $d(x,p_k)< S$, and denote $\Gamma_x=h^{-1}\{h(x)\}$ as the level set of $h$ at $h(x)$.  

Using the estimates in Lemma \ref{lem:CC}, we obtain:
\begin{align*}
    &\left|d(x,p_k)^2-d(x',p_k)^2-(h(x)-h(x'))^2\right|\\
    &\le \left|d(x,x')^2-(h(x)-h(x'))^2\right|+\Psi\\
    &= \left|d(x,x')-|h(x)-h(x')|\right|\cdot \left(d(x,x')+|h(x)-h(x')|\right)+\Psi\\
    &\le \Psi(2d(x,x')+\Psi)+\Psi\\
    &\le \Psi
\end{align*}
Let $\hat{p}_k\in\Gamma_x$ that minimizes the distance from $p_k$ among all the points in $\Gamma_x$. Then the same argument shows that
\begin{align*}
    \left|d(x,p_k)^2-d(x,\hat{p}_k)^2-(h(p_k)-h(\hat{p}_k))^2\right|\le \Psi
\end{align*}
Given that $h(x)=h(\hat{p}_k)$ and $h(p_k)=h(x')$, we deduce:
\begin{equation}{\label{eqn:rectangle}}
    \left|d(x,\hat{p}_k)^2-d(x',p_k)^2\right|\le \Psi
\end{equation}
We know that $|h(x)|\le S+\Psi$ and $[-2S+\Psi,2S-\Psi]\subset h(\gamma_{t_k-2S,t_k+2S})$. Therefore by choosing $r_0$ large enough, one can make sure that for $R\ge r_0$ we would have 
$$|h(x)|<\frac{3}{2}S,\qquad [\frac{3}{2}S,\frac{3}{2}S]\subset h(\gamma_{t_k-2S,t_k+2S})$$
Thus we could consider some $\gamma(s)\in\gamma_{t_k-2S,t_k+2S}\cap \Gamma_x$. We find: 
\begin{align*}
    \left|d(p_k,\hat{p}_k)-d(p_k,\gamma(s))\right|
    &=\left|d(p_k,\hat{p}_k)-|s-t_k|\right|\\
    &\le \left|h(\hat{p}_k)-|s-t_k|\right|+\Psi\\
    &= \left|h(\gamma(s))-|s-t_k|\right|+\Psi\\
    &\le \Psi
\end{align*}
Using \eqref{eqn:pythagoras}, we obtain:
\[
d(\hat{p}_k,\gamma(s))^2\le d(p_k,\gamma(s))^2-d(p_k,\hat{p}_k)^2+\Psi\le \Psi
\]
Applying the triangle inequality and combining with \eqref{eqn:rectangle}, we find:
\begin{align*}
    d(x,\gamma(s))^2&\le \left(d(x,\hat{p}_k)+d(\hat{p}_k,\gamma(s))\right)^2\\
    &\le d(x,\hat{p}_k)^2+\Psi\\
    &\le d(x',p_k)^2+\Psi
\end{align*}
Finally, from Lemma \ref{lem:CLS-step-1}, we know $d(x',p_k)\le 3c$ for sufficiently large $r_0$. Replacing $r_0$ with a larger value if necessary, we conclude that if $R\ge r_0$, then
$d(x,\gamma{(s)})^2\le 16c^2$. This guarantees that $x\in N_{4c}(\gamma_{0,\ell})$.
\end{proof}
Now we fix $r_0$ to be the constant from Lemma \ref{lem:ball-around-geodesic} and we assume $\ell>2r_0$.
\begin{prop}{\label{prop:cover}}
    We have $N_{5c}(\gamma_{r_0,\ell-r_0})\subset N_{4c}(\gamma_{0,\ell})$.
\end{prop}
\begin{proof}
    Consider a point $x\in N_{5c}(\gamma_{r_0,\ell-r_0})$ and suppose
    $$d(x,\gamma(\hat{t}))=\inf_{t\in [r_0,\ell-r_0]}\gamma(t)$$
    for some $\hat{t}\in [r_0,\ell-r_0]$. Then by assumption, we have $d(x,\gamma(\hat{t}))< 5c$. Furthermore, since $|t_{k+1}-t_k|\le 2L$ by Lemma \ref{lem:break-down}, there exists some $t_k\in[r_0,\ell-r_0]$ such that $|\hat{t}-t_k|< 2L$. Applying the triangle inequality gives
    $$d(x,p_k)< d(x,\gamma(\hat{t}))+2L\le S$$
    Therefore, $x\in B_k$. By Lemma \ref{lem:ball-around-geodesic}, for $t_k\in [r_0,\ell-r_0]$, we have
    $B_k \subset N_{4c}(\gamma_{0,\ell})$.
\end{proof}
\begin{prop}{\label{prop:jump}}
    Given $x\in M$, let $\hat{t}\in[0,\ell]$ such that $d(x,\gamma(\hat{t}))=d(x,\gamma_{0,\ell})$.
    If $\hat{t}\in [r_0,\ell-r_0]$, then $d(x,\gamma(\hat{t}))\le 4c$.
\end{prop}
\begin{proof}
    Suppose the assertion is false and $d(x,\gamma(\hat{t}))>4c$ with $\hat{t}\in [r_0,\ell-r_0]$. According to Proposition \ref{prop:cover}, we would then have $d(x,\gamma(\hat{t}))>5c$. By continuity, there exists a point $\hat{x}$ on the geodesic segment joining $x$ to $\gamma(\hat{t})$ where $d(\hat{x},\gamma(\hat{t}))=5c$. Applying Proposition \ref{prop:cover} again, we find $d(\hat{x},\gamma_{0,\ell})\le 4c$. However, since $\hat{x}$ is on the minimizing geodesic between $x$ and $\gamma(\hat{t})$, we must have
    $$4c\ge d(\hat{x},\gamma_{0,\ell})=d(\hat{x},\gamma(\hat{t}))=5c$$
    This is a contradiction. 
\end{proof}
\begin{cor}{\label{cor:large-nbhd}}
     For any $s$ with $\ell>4s+2r_0$, one has
     $$N_{s}(\gamma_{2s+r_0,\ell-2s-r_0})\subset N_{4c}(\gamma_{0,\ell})$$
\end{cor}
\begin{proof}
    For $x\in N_{s}(\gamma_{2s+r_0,\ell-2s-r_0})$, let $\hat{t}\in[0,\ell]$ such that $d(x,\gamma(\hat{t}))=d(x,\gamma_{0,\ell})$.
    We must then have $\hat{t}\in [r_0,\ell-r_0]$. The claim now follows directly from Proposition \ref{prop:jump}.
\end{proof}
\begin{cor}{\label{cor:one-end}}
    For $x\notin N_{4c}(\gamma_{0,\ell})$, let $\hat{t}\in[0,\ell]$ such that $d(x,\gamma(\hat{t}))=d(x,\gamma)$.
    If $d(x,\gamma(t))>d(x,\gamma)$ for all $t>\ell-r_0$, then $\hat{t}\le r_0$. Therefore, $\hat{t}\le r_0$ if either of the following conditions holds:
    \begin{itemize}
        \item $d(x,\gamma(0))<\frac{\ell-r_0}{2}$.
        \item $d(x,\gamma(0))<d(q,\gamma(\ell))-r_0$.
    \end{itemize}   
\end{cor}
\begin{proof}
    Since $x\notin N_{4c}(\gamma_{0,\ell})$, Proposition \ref{prop:jump} implies that if $\hat{t}\notin [\ell-r_0,\ell]$, then $\hat{t}$ must be within $[0,r_0]$. Given $t>{\ell}-r_0$, we consider the following scenarios:\\
    Under the first assumption:
    \begin{align*}
        d(x,\gamma(t))\ge \ell-r_0-d(x,\gamma(0))> \frac{\ell-r_0}{2}>d(x,\gamma(0))
    \end{align*}
    Under the second assumption:
    $$d(x,\gamma(t))\ge d(x,\gamma(\ell))-r_0>d(x,\gamma(0))$$
    In both cases, this ensures that $\hat{t}\in [0,r_0]$.
\end{proof}
\section{Proof of Theorem \ref{thm:effective}}
Now, let us fix $p\in M$ and $r>0$. We assume $R_g\ge 1$ within $B_r(p)$. We consider a minimizing geodesic $\gamma:[0,r] \ra M$ with $\gamma(0)=p$. Let $r_0$ be the universal constant from Lemma \ref{lem:ball-around-geodesic}, and without loss of generality, we assume that $r_0>10S$ and $r>32r_0$.   

We define the region
$$U=N_{4c}(\gamma_{0,r})\cup B_{6r_0}(p),\qquad V=B_{\frac{r}{16}}(p)\backslash U$$
\begin{lemma}{\label{lem:first-region}}
    We have $\vol(U)\le C(r_0,c) r$.
\end{lemma}
\begin{proof}
    The tubular neighborhood $N_{4c}(\gamma_{0,r})$ can be covered by $r\slash c$ geodesic balls of radius $4c$. Bishop-Gromov then implies that
    \begin{align*}
        \vol(U)\le \vol\left(N_{4c}(\gamma_{0,r})\right)+\vol(B_{6r_0}(p))\le C(\frac{r}{c}) c^3+Cr_0^3\le Cr
    \end{align*}
    where $C$ depends only on $c$ and $r_0$.
\end{proof}
Let us from now assume that $V\ne\emptyset$. For any $q\in V$, we let $\gamma^q:[0,\ell^q]$ to be the unit speed minimizing geodesic joining $q$ and $\gamma(\frac{r}{4})$. Note that it follows from the triangle inequality that $\frac{3}{16}r\le \ell^q \le \frac{5}{16}r$.
\begin{prop}{\label{prop:outside}}
    For any $q\in V$, there exists some $t^{q}\le\frac{r}{8}$ such that
    $$d(p,\gamma^q)=d(p,\gamma^q(t^{q}))\le 4c.$$
\end{prop}
\begin{proof}
First notice that $d(p,q)<\frac{r-r_0}{2}$ and $q\notin V$, hence Corollary \ref{cor:one-end} implies that
$$d(\gamma,\gamma^q(0))\ge d(p,q)-r_0>5c$$
Since $d(\gamma,\gamma^q(\ell^q))=0$, by continuity there exists some $\Tilde{t}\in [0,\ell^q]$ such that $d(\gamma,\gamma^q(\Tilde{t}))=5c$. It is clear that 
$$\min\{d\left(\gamma^q(\Tilde{t}),\gamma^q(0)\right),d\left(\gamma^q(\Tilde{t}),\gamma^q(\ell^q)\right)\}\le\frac{\ell^q}{2}$$
thus,
\begin{align*}
    d\left(p,\gamma^q(\Tilde{t})\right)\le \frac{\ell^q}{2}+\max\left\{d\left(p,\gamma^q(0)),d(p,\gamma^q(\ell^q)\right)\right\}\le\frac{13}{32}r<\frac{r-r_0}{2}
\end{align*}
Applying Corollary \ref{cor:one-end} again, we find that $\gamma^q(\Tilde{t})\in N_{5c}(\gamma_{0,r_0})$. We therefore conclude that $d(p,\gamma^q(\Tilde{t})) < \frac{3}{2}r_0$. This implies for all $t\in[0,\ell^q]$, we have
\[
    d(p,\gamma^q({t}))\le d(p,\gamma^q(\Tilde{t}))+|t-\tilde{t}| < \frac{3}{2}r_0+|t-\tilde{t}|
\]
But by assumption $d(p,\gamma^q(0)),d(p,\gamma^q(\ell^{q}))\ge 6r_0$, and hence $\tilde{t}\in [\frac{9}{2}r_0,\ell^q-\frac{9}{2}r_0]$ and $p\in N_{\frac{3}{2}r_0}\left(\gamma^q_{4r_0,\ell^q-4r_0}\right)$. Corollary \ref{cor:large-nbhd} then implies that 
$$d(p,\gamma^q(t^{q})):=\min_{t\in[0,\ell^q]} d(p,\gamma^q(t)) \le 4c.$$ 
To control $t^{q}$, we observe that
\[\ell^q-t^{q}\ge d(p,\gamma^q(\ell^q))-d(p,\gamma^q(t^{q}))= \frac{r}{4}-d(p,\gamma^q(t^{q}))\]
But given $\ell^q\le\frac{5}{16}r$ and $d(p,\gamma^q(t^{q}))\le 4c$, 
it follows that $t^{q}\le\frac{r}{8}$.
\end{proof}
Let us now pick some $q_0\in V$
such that $d(q,\gamma(\frac{r}{4}))\le d(q_0,\gamma(\frac{r}{4}))+r_0$ for all $q\in V$.
\begin{prop}
    We have $V\subset  N_{4c}(\gamma^{q_0}_{0,\ell^{q_0}})\cup B_{6r_0}(q_0)$.
\end{prop}
\begin{proof}
    Suppose the statement is false and there exists some $q\in V$ such that $q\notin  N_{4c}(\gamma^{q_0}_{0,\ell^{q_0}})\cup B_{6r_0}(q_0)$. Let $\hat{t}\in [0,\ell^{q_0}]$ so that $d(q,\gamma^{q_0}(\hat{t}))=d(q,\gamma^{q_0})$. We first observe that
    \[
    d(q_0,q)< 2\cdot\frac{r}{16}<  \ell^q-r_0=d(q,\gamma^{q_0}(\ell^{q_0}))-r_0
    \]
    Then, with $q\notin N_{4c}(\gamma^{q_0}_{0,\ell^{q_0}})$, Corollary \ref{cor:one-end} would imply that $\hat{t}\in[0,r_0]$. This implies:
    $$d(q,\gamma^{q_0})\ge d(q,q_0)-d(q_0,\gamma^{q_0}(\hat{t}))\ge d(q_0,q)-r_0$$
    As we are assuming $q\notin B_{6r_0}(q_0)$, it follows that $d(q,\gamma^{q_0})>5r_0$.
    By continuity, along $\gamma^{q}$ there exists some $t^*\in[0,\ell^q]$ such that
    $$t^*:=\inf\{t\in [0,\ell^q]:d(\gamma^{q}(t),\gamma^{q_0})=r_0\}$$
    By Proposition \ref{prop:outside}, there exists some $t^{q_0},t^q\in[0,\frac{r}{8}]$ satisfying
    $$d(p,\gamma^{q_0}(t^{q_0})) \le 4c,\qquad d(p,\gamma^{q}(t^q))\le 4c$$
    This implies $d(\gamma^{q_0}(t^{q_0}),\gamma^{q}(t^q))\le 8c<r_0$, thus the minimality assumption of $t^*$ gives $t^* \le \frac{r}{8}$. However, this shows that $\ell^{q}-t^*>\frac{r}{16}>2r_0$. Hence for all $t>\ell^{q_0}-r_0$
    \[d(\gamma^{q}(t^*),\gamma^{q_0}(t)))
        \ge  d(\gamma^{q}(t^*),\gamma^{q}(\ell^q))-r_0>r_0\]
    and we must have $\gamma^{q}(t^*)\in N_{r_0}(\gamma^{q_0}_{0,r_0})$ by Corollary \ref{cor:one-end}. In particular, this implies $d(q_0,\gamma^{q}(t^*))\le 2r_0$. Using the triangle inequality
    \begin{align*}
        d\left(q,\gamma(\frac{r}{4})\right)&\ge d(q,q_0)+d\left(q_0,\gamma(\frac{r}{4})\right)-2d\left(q_0,\gamma^{q}(t^*)\right)\\
        &\ge d(q,q_0)+d\left(q_0,\gamma(\frac{r}{4})\right)-4r_0
    \end{align*}
    But the assumption of $q_0$ implies $d\left(q_0,\gamma(\frac{r}{4})\right)\ge d\left(q,\gamma(\frac{r}{4})\right)-r_0$ and we obtain that $d(q,q_0) \le 5r_0$. This contradicts with the assumption that $q\notin B_{6r_0}(q_0)$.
\end{proof}
\begin{cor}{\label{cor:volume}}
    We have $\vol(V)\le Cr$, and hence $\vol(B_{\frac{r}{16}}(p))<Cr$.
\end{cor}
\begin{proof}
    This follows from the same argument as the proof of Lemma \ref{lem:first-region}.
\end{proof}
Now, we will prove the main result of this note.
\begin{proof}[Proof of Theorem \ref{thm:effective}]
    Given $p\in M$ and $r>0$, we suppose that $r>64r_0$. Otherwise, we have $\vol(B_{r}(p))\le Cr_0^3\le C r_0$. We can also assume that there exists a unit speed minimizing geodesic $\gamma:[0,r]\to M$ with $\gamma(0)=p$. It then follows from Corollary \ref{cor:volume} and the Bishop-Gromov volume comparison theorem again
    $$\vol(B_r(p))\le C\vol(B_{\frac{r}{16}}(p))\le Cr$$
    This proves Theorem \ref{thm:effective}.
\end{proof}

\end{document}